\documentclass[a4paper,10pt,twoside,draft]{article}
\pdfoutput=1

\makeatletter

\usepackage[T1]{fontenc}
\usepackage[utf8]{inputenc}
\usepackage{textcomp}

\usepackage[english]{babel}

\usepackage{enumitem}

\usepackage{color}

\usepackage{amsmath}
\usepackage{amsthm}
\usepackage{mathtools}

\usepackage{url}
\usepackage[unicode,colorlinks,pdfusetitle,final]{hyperref}

\usepackage{tikz}
\usetikzlibrary[arrows,decorations.pathreplacing,positioning,shapes,fadings]

\usepackage[numbers,sort]{natbib}

\usepackage{lmodern}
\usepackage{upgreek}
\usepackage{amssymb}
\usepackage{amsfonts}

\usepackage[scaled]{helvet}

\usepackage[bf,small]{titlesec}
\usepackage[labelfont={bf},labelsep=period]{caption}

\usepackage[final]{microtype}

\setenumerate{label=\alph*)}

\titlelabel{\thetitle.\space}

\numberwithin{equation}{section}

\newcounter{and}
\newdimen\instindent
\newcommand{\institute}[1]{\newcommand{\@institute}{#1}}
\newcommand{\inst}[1]{\unskip\smash{$^#1$}\setcounter{and}{1}\ignorespaces}
\newcommand{\email}[1]{\href{mailto:#1}{\texttt{#1}}}

\renewcommand{\maketitle}{
  { 
    \raggedright
    \Large
    \noindent
    \bfseries
    \@title
    \par
  }

  \vspace{1.5\baselineskip}

  { 
    \raggedright
    \renewcommand{\and}{\unskip, \ignorespaces}
    \noindent\ignorespaces\@author\par
  }

  \vspace{0.5\baselineskip}

  { 
    \small
    \parindent=0pt
    \parskip=0pt
    \setcounter{and}{1}%
    \renewcommand{\and}{%
      \par\stepcounter{and}%
      \hangindent\instindent
      \noindent
      \hbox to \instindent{\hss\smash{$^{\theand}$\enspace}}\ignorespaces
    }
    \setbox0=\vbox{\@institute}
    \ifnum\value{and}>9\relax\setbox0=\hbox{$^{88}$\enspace}%
    \else\setbox0=\hbox{$^{8}$\enspace}\fi
    \instindent=\wd0\relax
    \ifnum\value{and}=1\relax
    \else
      \setcounter{and}{1}%
      \hangindent\instindent
      \noindent
      \hbox to \instindent{\hss\smash{$^{\theand}$}\enspace}\ignorespaces
    \fi
    \ignorespaces
    \@institute\par
  }
}

\renewenvironment{abstract}{
  \addvspace{1.5\baselineskip}
  \topsep=0pt\partopsep=0pt
  \trivlist\item[\hskip\labelsep\bfseries Abstract.]
}{}

\newenvironment{acknowledgments}{
  \addvspace{1.5\baselineskip}
  \topsep=0pt\partopsep=0pt
  \trivlist\item[\hskip\labelsep\bfseries Acknowledgments.]
}{}

\newcommand{\latinabbr}{\textit}
\newcommand{\cf}{\latinabbr{cf.}\ }
\newcommand{\eg}{\latinabbr{e.g.}}

\newcommand{\ie}{\latinabbr{i.e.}}

\newcommand{\defn}{\doteq} 

\newcommand{\e}{\mathrm{e}}

\newcommand{\field}[1][K]{\mathbb{#1}}
\newcommand{\NN}{\field[N]}
\newcommand{\ZZ}{\field[Z]}
\newcommand{\RR}{\field[R]}
\newcommand{\CC}{\field[C]}

\DeclarePairedDelimiter{\abs}{\lvert}{\rvert}

\newcommand{\dif}{\mathrm{d}}
\newcommand{\od}[3][]{\frac{\dif^{#1}#2}{\dif#3^{#1}}}
\newcommand{\pd}[3][]{\frac{\partial^{#1}#2}{\partial#3^{#1}}}

\theoremstyle{plain}
\newtheorem{theorem}{Theorem}

\newtheorem{proposition}[theorem]{Proposition}

\theoremstyle{definition}
\newtheorem{definition}[theorem]{Definition}

\theoremstyle{remark}

\definecolor{hypercolor}{rgb}{0,0.2,0.7}
\hypersetup{
  linkcolor=hypercolor,
  urlcolor=hypercolor,
  citecolor=hypercolor,
  pdfauthor=Daniel Siemssen,
  pdfcreator=LaTeX,
  pdfproducer=pdfTeX
}

\newcommand{\lddots}{\,.\,.\,}

\usepackage{xstring}

\newcommand{\perm}[1]{\textls*[100]{#1}}
\newcommand{\cperm}[1]{%
  \StrLeft{#1}{1}[\cperm@first]%
  \StrRight{#1}{1}[\cperm@last]%
  \StrGobbleLeft{#1}{1}[\cperm@tmp]%
  \StrGobbleRight{\cperm@tmp}{1}[\cperm@middle]%
  $\dot\cperm@first$\textls[100]{\cperm@middle}$\dot\cperm@last$%
}

\newcommand{\perms}{\mathfrak{S}}
\newcommand{\cperms}{\mathfrak{C}}
\newcommand{\aperms}{\mathfrak{A}}

\newcommand{\oeis}[1]{\href{http://oeis.org/#1}{\texttt{#1}}}

\DeclareMathOperator{\ord}{ord}

\makeatother

\begin{document}

\title{Enumerating Permutations by their Run Structure}

\author{Christopher J. Fewster\inst{1} \and Daniel Siemssen\inst{2}}

\institute{
  Department of Mathematics, University of York, Heslington, York YO10 5DD, UK.\\
  E-Mail: \email{chris.fewster@york.ac.uk}
  \and
  Dipartimento di Matematica, Università di Genova, Via Dodecaneso 35, 16146 Genova, Italy.\\
  E-Mail: \email{siemssen@dima.unige.it}
}

\maketitle


\begin{abstract}
  Motivated by a problem in quantum field theory, we study the up and down structure of circular and linear permutations.
  In particular, we count the length of the (alternating) runs of permutations by representing them as monomials and find that they can always be decomposed into so-called `atomic' permutations introduced in this work.
  This decomposition allows us to enumerate the (circular) permutations of a subset of $\mathbb{N}$ by the length of their runs.
  Furthermore, we rederive, in an elementary way and using the methods developed here, a result due to Kitaev on the enumeration of valleys.
\end{abstract}


\section{Introduction}
\label{sec:introduction}

Let us adopt the following notation for integer intervals: $[a \lddots b] \defn [a, b] \cap \NN = \{a, a+1, \dotsc, b\}$ with the special case $[n] \defn [1 \lddots n]$.

Let $S \subset \NN$ contain $n$ elements. A \emph{permutation} of $S$ is a linear ordering $\sigma_1, \sigma_2, \dotsc, \sigma_n$ of the elements of $S$.
We denote the set of all $n!$ permutations of~$S$ by~$\perms_S$ and by~$\perms_n$ if $S = [n]$.
Almost all statements and proofs below are given for the special case $S = [n]$ and the obvious generalization to generic finite subsets of $\NN$ is implicit.
It is customary to write the permutation as the word $\sigma = \sigma_1 \sigma_2 \,\cdots\, \sigma_n$ and we can identify the permutation~$\sigma$ with the bijection $\sigma(i) = \sigma_i$.

Replacing the linear ordering by a circular one, we arrive at the notion of \emph{circular permutations} of~$S$, \ie, the arrangements of the elements $\sigma_1, \sigma_2, \dotsc, \sigma_n$ of $S$ around an oriented circle (turning the circle over generally produces a different permutation).
In this case one can always choose $\sigma_1$ to be the smallest element of~$S$ so that $\sigma_1 = 1$ if $S = [n]$.
It is clear that the set $\cperms_S$ of all circular permutations of $S$ contains $(n-1)!$ elements; if $S = [n]$, we denote it by $\cperms_n$.
Circular permutations may be written as words of the form $\sigma = \dot \sigma_1 \sigma_2 \,\cdots\, \dot \sigma_n$, where we emphasize the circular symmetry by a dot above outermost characters of the word.\footnote{This is in analogy with the notation for repeating decimals when representing rational numbers.}
Moreover, due to the circular symmetry we can define $\sigma(0) \defn \sigma(n)$ and $\sigma(n+1) \defn \sigma(1)$.

Let us introduce yet another class of permutations.
We denote by $\aperms^\pm_S \subset \perms_S$ the rising and falling \emph{`atomic'} permutations\footnote{The rationale for this naming should become clear later when we see that arbitrary permutations can be decomposed into atomic permutations, but no further.} of an $n$-element subset $S \subset \NN$.
We define $\aperms^+_S$ as those permutations that, when written as a word, begin with the smallest and end with the largest element of $S$.
The falling atomic permutations $\aperms^-_S$ are reversed rising atomic permutations, \ie, they begin with the largest element of $S$ and end with the smallest.
Naturally, the cardinality of $\aperms^\pm_S$ is $(n-2)!$.
If $S = [n]$, we write $\aperms^\pm_n$ and see that it is the set of permutations of the form $1\,\dotsm\,n$ (\ie, $\sigma_1 = 1$ and $\sigma_n = n$) or $n\,\dotsm\,1$ (\ie, $\sigma_1 = n$ and $\sigma_n = 1$) respectively.

We say that $i$ is a \emph{descent} of $\sigma$ if $\sigma(i) > \sigma(i+1)$ and, conversely, that it is an \emph{ascent} of a (circular) permutation $\sigma$ if $\sigma(i) < \sigma(i+1)$.
For example, the permutation \perm{52364178} has the descents $1, 4, 5$ and the ascents $2, 3, 6, 7$ whereas the circular permutation \cperm{14536782} has the descents $3, 7, 8$ and the ascents $1, 2, 4, 5, 6$.
We can collect all the descents of a (circular) permutation $\sigma$ in the \emph{descent set}
\begin{equation*}
  D(\sigma) \defn \{ i \mid \sigma(i) > \sigma(i+1) \}.
\end{equation*}
It is an elementary exercise in enumerative combinatorics to count the number of permutations of~$[n]$ whose descent set is given by a fixed $S \subseteq [n-1]$.
Let $S = \{s_1, s_2, \dotsc, s_k \}$ be an ordered subset of~$[n-1]$, then
\begin{equation}\label{eq:num_permations_descent_set}
  \beta(S)
  \defn \big|\{ \sigma \in \perms_n \mid D(\sigma) = S \}\big|
  = \sum_{T \subseteq S} (-1)^{|S-T|} \binom{n}{s_1, s_2 - s_1, s_3 - s_2, \dotsc, n - s_k},
\end{equation}
see, for example, \cite[Theorem 1.4]{bona:2012}.
This result is easily adapted to circular permutations.

Related to the notions of \emph{ascents} and \emph{descents} are the concepts of \emph{peaks} and \emph{valleys}:
A peak occurs at position $i$ if $\sigma(i-1) < \sigma(i) > \sigma(i+1)$, whereas a valley occurs in the opposite situation $\sigma(i-1) > \sigma(i) < \sigma(i+1)$.
The peaks and valleys of a (circular) permutation $\sigma$ split it into \emph{runs}, also called \emph{sequences} or \emph{alternating runs}.

\begin{definition}
  A \emph{run}~$r$ of a (linear or circular) permutation $\sigma$ is an interval $[i \lddots j]$ such that $\sigma(i) \gtrless \sigma(i+1) \gtrless \dotsb \gtrless \sigma(j)$ is a monotone sequence, either increasing or decreasing, and so that it cannot be extended in either direction; its length is defined to be $j-i$.
  If $\sigma$ is a permutation of an $n$-element set, the collection of the lengths of all runs gives a partition~$p$ of $n-1$ (linear permutations) or $n$ (circular permutations).
  The partition $p$ is called the \emph{run structure} of~$\sigma$.
\end{definition}

It follows that a run starts and ends at peaks, valleys or at the outermost elements of a permutation.
For example, the permutation \perm{52364178} has runs $[1 \lddots 2]$, $[2 \lddots 4]$, $[4 \lddots 6]$, $[6 \lddots 8]$ with lengths $1, 2, 2, 2$, whereas the circular permutation \cperm{14536782} has runs $[1 \lddots 3]$, $[3 \lddots 4]$, $[4 \lddots 7]$, $[7 \lddots 9]$, of lengths $2, 1, 3, 2$.
Representing these runs by their image under the permutation, they are more transparently written as $\textls[100]{52, 236, 641, 178}$ and $\textls[100]{145, 53, 3678, 821}$ respectively.
The runs of (circular) permutations can also be neatly represented as directed graphs as shown in Figure~\ref{fig:run_graphs}.
In these graphs the peaks and valleys correspond to double sinks and double sources.
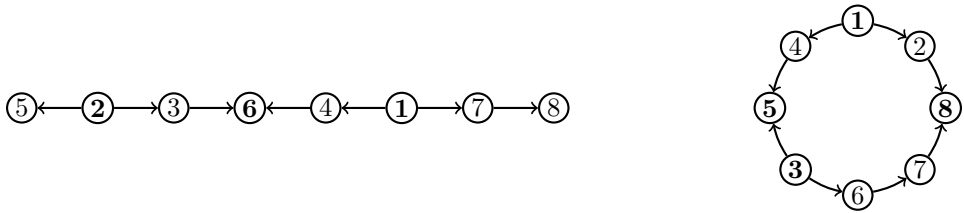
\begin{figure}
  \centering
  \begin{tikzpicture}
    [ scale=1
    , every node/.style={circle, draw, inner sep=1pt, minimum size=11pt}
    , every path/.style={thick}
    ]


    \node (5) at (0,0) {$5$};
    \node (2) at (1,0) {$\mathbf{2}$} edge [->] (5);
    \node (3) at (2,0) {$3$}          edge [<-] (2);
    \node (6) at (3,0) {$\mathbf{6}$} edge [<-] (3);
    \node (4) at (4,0) {$4$}          edge [->] (6);
    \node (1) at (5,0) {$\mathbf{1}$} edge [->] (4);
    \node (7) at (6,0) {$7$}          edge [<-] (1);
    \node (8) at (7,0) {$8$}          edge [<-] (7);


    \node[name=c, circle, minimum size=65pt, draw=none] at (11,0) {};

    \node (1) at (c.north)      {$\mathbf{1}$};
    \node (4) at (c.north west) {$4$}          edge [<-, bend left=10]  (1);
    \node (5) at (c.west)       {$\mathbf{5}$} edge [<-, bend left=10]  (4);
    \node (3) at (c.south west) {$\mathbf{3}$} edge [->, bend left=10]  (5);
    \node (6) at (c.south)      {$6$}          edge [<-, bend left=10]  (3);
    \node (7) at (c.south east) {$7$}          edge [<-, bend left=10]  (6);
    \node (8) at (c.east)       {$\mathbf{8}$} edge [<-, bend left=10]  (7);
    \node (2) at (c.north east) {$2$}          edge [->, bend left=10]  (8)
                                               edge [<-, bend right=10] (1);
  \end{tikzpicture}

  \caption{The two directed graphs representing the runs of the permutation \protect\perm{52364178} (left) and the circular permutation \protect\cperm{14536782} (right). Peaks and valleys are indicated by boldface numbers.}
  \label{fig:run_graphs}
\end{figure}

Motivated by a problem in mathematical physics \cite{fewster:2012b} (see also section~\ref{sec:app}), we are interested in the following issue, which we have not found discussed in the literature.
By definition, the run structure associates each permutation $\sigma \in \cperms_n$ with a partition~$p$ of~$n$.
For example, \cperm{14536782} and \cperm{13452786} both correspond to the same partition $1 + 2 + 2 + 3$ of $8$.
Our interest is in the inverse problem: given a partition $p$ of $n$, we ask for the number~$Z_{\cperms}(p)$ of circular permutations whose run structure is given by~$p$.
One may consider similar questions for other classes of permutations, with slight changes; for example, note that the run structure of a permutation $\sigma \in \perms_n$ is a partition of $n-1$.

The literature we have found focuses on issues such as the number of permutations in which all runs have length one, \eg, André's original study~\cite{andre:1895}, or on the enumeration question where the order of run lengths is preserved, so obtaining a map to compositions, rather than partitions, of $n$.
See, for example, \cite{brown:2007}, which studies this for the ordinary (noncircular) permutations.
In principle our question could be obtained by specialising this more difficult problem, but here we take a direct approach.
Another alternative approach not followed here, would be to take \eqref{eq:num_permations_descent_set} as a starting point and to sum over all descent sets corresponding to a particular run structure.
However, this would neither be straightforward from the theoretical side, nor would it make for efficient computation.
By contrast, our direct method was designed to facilitate computation; for the application in \cite{fewster:2012b} calculations were taken up to $n=65$ using exact integer arithmetic in Maple\texttrademark~\cite{maple:16}.

Sections~\ref{sec:atomic}, \ref{sec:circular} and \ref{sec:linear} deal, respectively, with the enumeration of the run structure of atomic, circular and linear permutations.
Using a suitable decomposition, this is accomplished in each case by reducing the enumeration problem to that for atomic permutations.
In section~\ref{sec:counting_valleys} we apply and extend the methods developed in the preceding sections to enumerate the valleys of permutations, thereby reproducing a result of Kitaev~\cite{kitaev:2007}.
Finally, in section~\ref{sec:app}, we discuss the original motivation for our work and other applications.


\section{Atomic permutations}
\label{sec:atomic}

Let us first discuss the significance of the atomic permutations.
We say that a permutation $\sigma \in \perms_S$ of~$S$ \emph{contains an atomic permutation} $\pi \in \aperms_T$, $T \subset S$, if $\pi$ can be considered a subword of~$\sigma$.
The atomic permutation~$\pi$ in~$\sigma$ is called \emph{inextendible} if $\sigma$ contains no other atomic permutation $\pi' \in \aperms_{T'}$, $T' \subset S$, such that $T \subsetneq T'$.

In particular, any permutation $\sigma \in \perms_S$ of~$S$ with $\abs{S} \geq 2$ contains an inextendible atomic permutation $\pi \in \perms_T$ of a subset $T \subset S$ that contains both the smallest and the largest element of~$S$.
That is, if $S = [n]$ and we consider $\sigma$ as a word, it contains a subword~$\pi$ of the form $1 \,\dotsm\, n$ or $n \,\dotsm\, 1$.
The permutation $\pi$ will be called the \emph{principal atom} of~$\sigma$.

\begin{proposition}\label{prop:decomposition}
  Any permutation $\sigma \in \perms_S$ of a finite set $S \subset \NN$ can be uniquely decomposed into a tuple $(\pi^1, \dotsc, \pi^k)$ of inextendible atomic permutations $\pi^i \in \aperms_{T_i}$, $T_i \subset S$ (non-empty) such that $\pi^i_{\abs{T_i}} = \pi^{i+1}_1$ for all $i < k$ and $\cup_i T_i = S$.
  We call $\pi^i$ the \emph{atoms} of $\sigma$.
\end{proposition}
\begin{proof}
  \emph{Existence:}
  It is clear that any permutation of a set of $1$~or $2$~elements is an atomic permutation.
  Suppose, for some $n\ge 3$, that all permutations of $n-1$ elements or less can be decomposed into inextendible atomic permutations.
  Without loss of generality, we show that any non-atomic permutation $\sigma \in \perms_n$ also has a decomposition into inextendible atomic permutations.
  Regarding $\sigma$ as a word, we can write $\sigma = \alpha \cdot n \cdot \omega$, where $\alpha$ and $\omega$ are non-empty subwords.
  Notice that the permutations $\alpha \cdot n$ and $n \cdot \omega$ have a unique decomposition by assumption.
  Since an atomic permutation begins or ends with the largest element, we find that a decomposition of~$\sigma$ into inextendible atomic permutations is given by the combination of the decompositions of $\alpha \cdot n$ and $n \cdot \omega$.

  \emph{Uniqueness:} This is clear from the definition of inextendibility.
\end{proof}

We now begin the enumeration of atomic permutations according to their run structure.
That is, for every partition~$p$ of ${n-1}$ we aim to find the number $Z_{\aperms}(p)$ of atomic permutations~$\aperms^\pm_n$ of length~$n$.

Observe that any $\sigma \in \aperms^+_{n}$ can be extended to a permutation in $\aperms^+_{n+1}$ by replacing $n$ by $n+1$ and reinserting $n$ in any position after the first and before the last.
Thus, \perm{13425} can be extended to \perm{153426}, \perm{135426}, \perm{134526} or \perm{134256}.
Every permutation in $\aperms^+_{n+1}$ arises in this way, as can be seen by reversing the procedure.
The effect on the run lengths can be described as follows.

\begin{description}[leftmargin=*]
  \item[Case 1:]\phantomsection\label{item:case_1}
    The length of one of the runs can be increased by one by inserting $n$ either at
    \begin{enumerate}[leftmargin=*]
      \item
        the end of an increasing run if it does not end in $n+1$, thereby increasing its length (\eg, \perm{13425} $\to$ \perm{134526})
        \\[.5em]
        \begin{tikzpicture}
          [ scale=.82
          , every node/.style={circle, draw=black!60, fill=white, inner sep=1pt, minimum size=11pt}
          , every path/.style={thick}
          ]
          \tikzstyle{dashsolid} = [dash pattern=on 4.5mm off .5mm on .5mm off .5mm on .5mm off .5mm on .5mm off .5mm on 4.5mm]

          \draw[dotted, path fading=west] (0,0) -- (1,0) ++(0,1pt);
          \draw[dotted, path fading=east] (5,0) -- (6,0) ++(0,1pt);
          \node (A) at (1,0) {};
          \node (B) at (2,0) {} edge [->] (A);
          \node (C) at (4,0) {} edge [<-, dashsolid] (B);
          \node (D) at (5,0) {} edge [->] (C);

          \draw[->] (6.5,0) -- (7,0);

          \draw[dotted, path fading=west] (7.5,0) -- (8.5,0) ++(0,1pt);
          \draw[dotted, path fading=east] (13.5,0) -- (14.5,0) ++(0,1pt);
          \node (A) at (8.5,0) {};
          \node (B) at (9.5,0) {} edge [->] (A);
          \node (C) at (11.5,0) {} edge [<-, dashsolid] (B);
          \node[draw=black] (N) at (12.5,0) {$n$} edge [<-] (C);
          \node (D) at (13.5,0) {} edge [->] (N);
        \end{tikzpicture}
      \item
        the penultimate position of an increasing run, thereby increasing its own length if it ends in $n+1$ (\eg, \perm{13425} $\to$ \perm{135426}) or increasing the length of the following decreasing run otherwise (\eg, \perm{13425} $\to$ \perm{134256})
        \\[.5em]
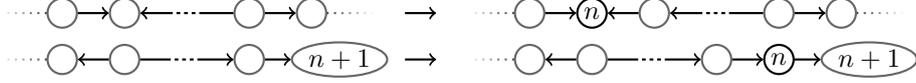
        \begin{tikzpicture}
          [ scale=.82
          , every node/.style={circle, draw=black!60, fill=white, inner sep=1pt, minimum size=11pt}
          , every path/.style={thick}
          ]
          \tikzstyle{dashsolid} = [dash pattern=on 4.5mm off .5mm on .5mm off .5mm on .5mm off .5mm on .5mm off .5mm on 4.5mm]

          \draw[dotted, path fading=west] (0,0) -- (1,0) ++(0,1pt);
          \draw[dotted, path fading=east] (5,0) -- (6,0) ++(0,1pt);
          \node (A) at (1,0) {};
          \node (B) at (2,0) {} edge [<-] (A);
          \node (C) at (4,0) {} edge [->, dashsolid] (B);
          \node (D) at (5,0) {} edge [<-] (C);

          \draw[->] (6.5,0) -- (7,0);

          \draw[dotted, path fading=west] (7.5,0) -- (8.5,0) ++(0,1pt);
          \draw[dotted, path fading=east] (13.5,0) -- (14.5,0) ++(0,1pt);
          \node (A) at (8.5,0) {};
          \node[draw=black] (N) at (9.5,0) {$n$} edge [<-] (A);
          \node (B) at (10.5,0) {} edge [->] (N);
          \node (C) at (12.5,0) {} edge [->, dashsolid] (B);
          \node (D) at (13.5,0) {} edge [<-] (C);

          \draw[dotted, path fading=west] (0,-0.75) -- (1,-0.75) ++(0,1pt);
          \node (A) at (1,-0.75) {};
          \node (B) at (2,-0.75) {} edge [->] (A);
          \node (C) at (4,-0.75) {} edge [<-, dashsolid] (B);
          \node[ellipse] (D) at (5.45,-0.75) {$n+1$} edge [<-] (C);

          \draw[->] (6.5,-0.75) -- (7,-0.75);

          \draw[dotted, path fading=west] (7.5,-0.75) -- (8.5,-0.75) ++(0,1pt);
          \node (A) at (8.5,-0.75) {};
          \node (B) at (9.5,-0.75) {} edge [->] (A);
          \node (C) at (11.5,-0.75) {} edge [<-, dashsolid] (B);
          \node[draw=black] (N) at (12.5,-0.75) {$n$} edge [<-] (C);
          \node[ellipse] (D) at (13.95,-0.75) {$n+1$} edge [<-] (N);
        \end{tikzpicture}
    \end{enumerate}
  \item[Case 2:]\phantomsection\label{item:case_2}
    Any run of length $i+j \geq 2$ becomes three run of lengths $1$, $i$ and $j$ if we insert $n$ either after
    \begin{enumerate}[leftmargin=*]
      \item
        $i$ elements of an increasing run (\eg, \perm{13425} $\to$ \perm{153426} exemplifies $i=1$, $j=1$)
        \\[.5em]\hspace*{-15pt}
        \begin{tikzpicture}
          [ scale=.82
          , every node/.style={circle, draw=black!60, fill=white, inner sep=1pt, minimum size=11pt}
          , every path/.style={thick}
          ]
          \tikzstyle{dashsolid} = [dash pattern=on 4.5mm off .5mm on .5mm off .5mm on .5mm off .5mm on .5mm off .5mm on 4.5mm]
          \tikzstyle{shortdashsolid} = [dash pattern=on 2.2mm off .5mm on .5mm off .5mm on .5mm off .5mm on .5mm off .5mm on 2.2mm]

          \draw[dotted, path fading=west] (0,0) -- (1,0) ++(0,1pt);
          \draw[dotted, path fading=east] (5,0) -- (6,0) ++(0,1pt);
          \node (A) at (1,0) {};
          \node (B) at (2,0) {} edge [->] (A);
          \node (C) at (4,0) {};
          \path[->, dashsolid] (B) edge node[above, rectangle, draw=none, fill=none, minimum size=0] {\footnotesize$i+j$} (C);
          \node (D) at (5,0) {} edge [->] (C);

          \draw[->] (6.5,0) -- (7,0);

          \draw[dotted, path fading=west] (7.5,0) -- (8.5,0) ++(0,1pt);
          \draw[dotted, path fading=east] (14.5,0) -- (15.5,0) ++(0,1pt);
          \node (A) at (8.5,0) {};
          \node (B) at (9.5,0) {} edge [->] (A);
          \node[draw=black] (N) at (11,0) {$n$};
          \path[->, shortdashsolid] (B) edge node[above, rectangle, draw=none, fill=none, minimum size=0] {\footnotesize$i\vphantom{j}$} (N);
          \node (M) at (12,0) {} edge [->] (N);
          \node (C) at (13.5,0) {};
          \path[->, shortdashsolid] (M) edge node[above, rectangle, draw=none, fill=none, minimum size=0] {\footnotesize$j$} (C);
          \node (D) at (14.5,0) {} edge [->] (C);
        \end{tikzpicture}
      \item
        $i+1$ elements of a decreasing run (\eg, \perm{14325} $\to$ \perm{143526} for $i=1$, $j=1$)
        \\[.5em]\hspace*{-15pt}
        \begin{tikzpicture}
          [ scale=.82
          , every node/.style={circle, draw=black!60, fill=white, inner sep=1pt, minimum size=11pt}
          , every path/.style={thick}
          ]
          \tikzstyle{dashsolid} = [dash pattern=on 4.5mm off .5mm on .5mm off .5mm on .5mm off .5mm on .5mm off .5mm on 4.5mm]
          \tikzstyle{shortdashsolid} = [dash pattern=on 2.2mm off .5mm on .5mm off .5mm on .5mm off .5mm on .5mm off .5mm on 2.2mm]

          \draw[dotted, path fading=west] (0,0) -- (1,0) ++(0,1pt);
          \draw[dotted, path fading=east] (5,0) -- (6,0) ++(0,1pt);
          \node (A) at (1,0) {};
          \node (B) at (2,0) {} edge [<-] (A);
          \node (C) at (4,0) {};
          \path[->, dashsolid] (C) edge node[above, rectangle, draw=none, fill=none, minimum size=0] {\footnotesize$i+j$} (B);
          \node (D) at (5,0) {} edge [<-] (C);

          \draw[->] (6.5,0) -- (7,0);

          \draw[dotted, path fading=west] (7.5,0) -- (8.5,0) ++(0,1pt);
          \draw[dotted, path fading=east] (14.5,0) -- (15.5,0) ++(0,1pt);
          \node (A) at (8.5,0) {};
          \node (B) at (9.5,0) {} edge [<-] (A);
          \node (M) at (11,0) {};
          \path[->, shortdashsolid] (M) edge node[above, rectangle, draw=none, fill=none, minimum size=0] {\footnotesize$i\vphantom{j}$} (B);
          \node[draw=black] (N) at (12,0) {$n$} edge [<-] (M);
          \node (C) at (13.5,0) {};
          \path[->, shortdashsolid] (C) edge node[above, rectangle, draw=none, fill=none, minimum size=0] {\footnotesize$j$} (N);
          \node (D) at (14.5,0) {} edge [<-] (C);
        \end{tikzpicture}
    \end{enumerate}
\end{description}
An analogous argument can be made for the falling atomic permutations $\aperms^-_n$.

Notice that partitions of positive integers can be represented by the monomials in the ring of polynomials\footnote{If one wants to encode also the order of the run (\eg, to obtain a map from permutations of length $n$ to the compositions of $n$), one can exchange the polynomial ring with a noncommutative ring. Alternatively, if one wants to encode the direction of a run, one could study instead the ring $\ZZ[x_1, y_1, x_2, y_2, \dotsc]$, where $x_i$ denotes an increasing run of length $i$ and $y_j$ encodes a decreasing run of length $j$.} $\ZZ[x_1, x_2, \dotsc]$ in infinitely many variables $x_1, x_2, \dotsc$ and with integer coefficients.
That is, we express a partition $p = p_1 + p_2 + \dotsb + p_m$ as $x_{p_1} x_{p_2} \,\cdots\, x_{p_m}$
 (\eg, the partition $1 + 2 + 2 + 3$ of $8$ is written as $x_1 x_2^2 x_3$).

Now, let $p$ be a partition of $n-1$ and $X$ the corresponding monomial.
To this permutation there correspond $Z_{\aperms}(p)$ permutations in $\aperms^\pm_n$ which can be extended to permutations in $\aperms^\pm_{n+1}$ in the manner described above.
Introducing the (formally defined) differential operator
\begin{equation}\label{eq:diff}
  \mathcal{D}   \defn \mathcal{D}_0 + \mathcal{D}_+
  \quad\text{with}\quad
  \mathcal{D}_0 \defn \sum_{i=1}^\infty x_{i+1} \pd{}{x_i},
  \;\;
  \mathcal{D}_+ \defn \sum_{i,j \,\geq\, 1} x_1 x_i x_j \pd{}{x_{i+j}},
\end{equation}
we can describe this extension in terms of the action of $\mathcal{D}$ on $X$.
We say that $\mathcal{D}_0$ is the \emph{degree-preserving} part of $\mathcal{D}$; it represents the \hyperref[item:case_1]{case 1} of increasing the length of a run: the differentiation $\partial/\partial x_i$ removes one of the runs of length $i$ and replaces it by a run of length $i+1$, keeping account of the number of ways in which this can be done. Similarly,
\hyperref[item:case_2]{case 2} of splitting a run into $3$ parts is represented by the \emph{degree-increasing} part $\mathcal{D_+}$.
For example, each of the $7$ atomic permutations corresponding to the partition $1 + 1 + 3$ can be extended as
\begin{equation*}
   \mathcal{D} x_1^2 x_3 = 2 x_1 x_2 x_3 + x_1^2 x_4 + x_1^4 x_2,
\end{equation*}
\ie, each can be extended to two atomic permutations corresponding to the partitions $1 + 2 + 3$, one corresponding to $1 + 1 + 4$ and one to $1 + 1 + 1 + 1 + 2$.

Therefore, starting from the trivial partition $1$ of $1$, represented as $x_1$, we can construct a recurrence relation for polynomials $A_n = A_n(x_1, x_2, \dotsc, x_n)$ which, at every step $n \geq 1$, encode the number of atomic permutations $Z_{\aperms}(p)$ of length $n+1$ with run structure given by a partition~$p$ of~$n$ as the coefficients of the corresponding monomial in $A_n$.
The polynomial~$A_n$, accordingly defined by
\begin{equation}\label{eq:A_rel_Z}
  A_n = \sum_{p \vdash n} Z_{\aperms}(p) \prod_{i=1}^n x_i^{p(i)},
\end{equation}
where the sum is over all partitions $p$ of $n$ and $p(i)$ denotes the multiplicity of $i$ in the partition $p$, can thus be computed from the recurrence relation
\begin{subequations}\label{eq:A_recurrence}
  \begin{align}
    A_1 & \defn x_1, \\
    A_n & \defn \mathcal{D} A_{n-1}, \qquad (n \geq 2).
  \end{align}
\end{subequations}
We say that the polynomials~$A_n$ enumerate the run structure of the atomic permutations.

We summarize these results in the following proposition:
\begin{proposition}\label{prop:Z_A}
  The number $Z_{\aperms}(p)$ of rising or falling atomic permutations of length $n-1$ corresponding to a given run structure (\ie, a partition $p$ of $n$), is determined by the polynomial $A_n$ via~\eqref{eq:A_rel_Z}.
  The polynomials $A_n$ satisfy the recurrence relation~\eqref{eq:A_recurrence}.
\end{proposition}
\noindent
Note that atomic permutations always contain an odd number of runs and thus $Z_{\aperms}(p)$ is zero for even partitions $p$.

It will prove useful to combine all generating functions $A_n$ into the formal series
\begin{equation*}
  \mathcal{A}(\lambda)
  \defn \sum_{n=0}^\infty A_{n+1} \frac{\lambda^n}{n!}
  = \sum_{n=0}^\infty \mathcal{D}^n A_1 \frac{\lambda^n}{n!},
\end{equation*}
which can be expressed compactly as the exponential
\begin{equation*}
  \mathcal{A}(\lambda) = \exp(\lambda \mathcal{D}) A_1.
\end{equation*}

The first few $A_n$ are given by
\begin{align*}
  A_2 & = x_2 \\
  A_3 & = x_3 + x_1^3 \\
  A_4 & = x_4 + 5 x_2 x_1^2 \\
  A_5 & = x_5 + 7 x_3 x_1^2 + 11 x_2^2 x_1 + 5 x_1^5 \\
  A_6 & = x_6 + 9 x_4 x_1^2 + 11 x_2^3 + 38 x_3 x_2 x_1 + 61 x_2 x_1^4.
\end{align*}
For example, we can read off from $A_5$ that there is $1$ permutation in $\aperms^\pm_6$ corresponding to the trivial partition $5 = 5$, $7$ corresponding to the partition $5 = 1 + 1 + 3$, $11$ corresponding to $5 = 1 + 2 + 2$ and $5$ corresponding to $5 = 1 + 1 + 1 + 1 + 1$.
As a check, we note that $1 + 7 + 11 + 5 = 24$, which is the total number of elements of $\aperms^\pm_6$; similarly, the coefficients in the expression for $A_6$ sum to $120$, the cardinality of $\aperms^\pm_7$. A direct check that the coefficients in $A_n$ sum to $(n-1)!$ for all $n$ will be given in the last paragraph of section \ref{sec:counting_valleys}.

We now consider what can be said concerning the degree $k$ part, $A^{(k)}_n$, of the polynomials $A_n$. The first degree term $A^{(1)}_n$ of $A_n$ is $x_n$ as can be seen by a trivial induction using  $A^{(1)}_n = \mathcal{D}_0 A^{(1)}_{n-1}$, which follows from the recurrence relation~\eqref{eq:A_recurrence}.
Therefore $Z_{\aperms}(n) = 1$.

For $A_n^{(k)}$ with $k > 1$ the effect of $\mathcal{D}_+$ has to be taken into account as well, complicating matters considerably.
Nevertheless, the general procedure is clear: once $A_m^{(k-2)}$ is known for all $m < n$, $A_n^{(k)}$ can be obtained as
\begin{equation*}
  A_n^{(k)} = \mathcal{D}_0 A_{n-1}^{(k)} + \mathcal{D}_+ A_{n-1}^{(k-2)} = \sum_{m=k-1}^{n-1} \mathcal{D}_0^{n-m-1} \mathcal{D}_+ A_m^{(k-2)}.
\end{equation*}
Here one can make use of the following relation.
Applying $\mathcal{D}_0$ repeatedly to any monomial $x_{i_1} x_{i_2} \dotsm x_{i_k}$ of degree $k$ yields, as a consequence of the Leibniz rule,
\begin{equation}\label{eq:repeated_D}
  \mathcal{D}_0^n x_{i_1} x_{i_2} \dotsm x_{i_k} = \sum_{\substack{j_1, j_2, \dotsc, j_k \,\geq\, 0 \\ j_1 + j_2 + \dotsb + j_k = n}} \binom{n}{j_1, j_2, \dotsc, j_k}\, x_{i_1 + j_1} x_{i_2 + j_2} \dotsm x_{i_k + j_k}.
\end{equation}

This observation provides the means to determine the third degree term $A_n^{(3)}$.
Applying $\mathcal{D}_+$ to any $A_m^{(1)} = x_m$ with $m \geq 2$ produces $x_1 x_p x_q$ with $p+q = m$ and $p, q \geq 1$.
Moreover, the repeated action of $\mathcal{D}_0$ on $x_1 x_p x_q$ is described by~\eqref{eq:repeated_D} and thus
\begin{equation*}
  A_n^{(3)} = \sum_{\substack{p, q, r, s, t \,\geq\, 0 \\ 1 + p + q + r + s + t = n}} \binom{n - p - q - 1}{r, s, t}\, x_{1+r} x_{p+s} x_{q+t}.
\end{equation*}
After some algebra this yields
\begin{proposition}
  The third degree term $A_n^{(3)}$ of the polynomial $A_n, n \geq 3,$ is given by
  \begin{equation}\label{eq:third_degree}
    A_n^{(3)}
    = \sum_{\substack{i, j, k \,\geq\, 1 \\ i + j + k = n}} \sum_{q=1}^{k} \frac{n-q-1}{n-q-j}\, \binom{n - q - 2}{i - 1, j - 1, k - q}\, x_i x_j x_k.
  \end{equation}
\end{proposition}

The equation~\eqref{eq:third_degree} for the third degree term $A_n^{(3)}$ can be rewritten into a formula for $Z_{\aperms}(p_1 + p_2 + p_3)$, \ie, the number of permutations of $[n+1]$ that start with $1$, end with $n+1$ and have three runs of lengths $p_1, p_2, p_3$, by changing the first sum to a sum over $i, j, k \in \{ p_1, p_2, p_3 \}$.
In particular, this gives rise to three integer series for the special cases
\begin{equation*}
  Z_{\aperms}(n+n+n), \quad Z_{\aperms}(1+n+n), \quad Z_{\aperms}(1+1+n),
\end{equation*}
with $n \in \NN$.

The first series
\begin{align*}
  Z_{\aperms}(n+n+n)
  & = \sum_{q=1}^n \frac{3n-q-1}{2n-q}\, \binom{3n-q-2}{n-1,n-1,n-q}\\
  & = 1, 11, 181, 3499, 73501, 1623467, \dotsc \qquad (n \geq 1)
\end{align*}
gives the number of atomic permutations with three runs of equal length $n$.
It does not appear to be known in the literature nor can it be found in the OEIS \cite{oeis} and the authors were not able to express it in a closed form.
For the second series, however, a simple closed form can be found:
\begin{align*}
  Z_{\aperms}(1+n+n)
  & = \sum_{q=1}^n \bigg( \binom{2n-q}{n-1} + \binom{2n-q-1}{n-1} \bigg) + \frac{1}{2}\, \binom{2n}{n}\\
  & = 2\, \binom{2n}{n} - 1
    = 11, 39, 139, 503, 1847, \dotsc, \qquad (n \geq 2)
\end{align*}
is the number of atomic permutations in $\aperms^\pm_{2n+2}$ with two runs of length $n$. One may understand this directly: there are $\binom{2n}{n}$
permutations in which the length $1$ run is between the others and $\binom{2n}{n}-1$ in which it is either first or last.
The third series, $Z_{\aperms}(1+1+n)$, \ie, the number of atomic permutations in $\aperms^\pm_{n+3}$ with two runs of length $1$, is given by the odd numbers bigger than $3$:
\begin{equation*}
  Z_{\aperms}(1+1+n) = 2n + 1 = 5, 7, 9, 11, 13, 15, \dotsc, \qquad (n \geq 2).
\end{equation*}

Observe that terms of the form $x_1^n$ in $A_n$ encode alternating permutations,
which were already investigated by André in the 1880's \cite{andre:1881}.
As a consequence of his results, we find that the alternating atomic permutations are enumerated by the \emph{secant numbers} $S_n$, the coefficients of the Maclaurin series of $\sec x = S_0 + S_1 x^2/2! + S_2 x^4/4! + \dotsb$,
\begin{equation*}
  Z_{\aperms}\bigg(\sum_{i=1}^{2n+1} 1\bigg) = S_n = 1, 1, 5, 61, 1385, 50521, \dotsc  \quad \text{($n \geq 0$, OEIS series \oeis{A000364})}.
\end{equation*}
This is due to the fact that all alternating atomic permutations of $[2n]$ can be understood as the reverse alternating permutations of $[2 \lddots 2n-1]$ with a prepended $1$ and an appended $2n$.
Moreover, since any $x_1^{2n+1}$ can only be produced through an application of $\mathcal{D}$ on $x_2 x_1^{2(n-1)}$, we also have $Z_{\aperms}\big(2 + \sum_{i=1}^{2(n-1)} 1\big) = S_n$.


\section{Circular permutations}
\label{sec:circular}

The methods developed in the last section to enumerate atomic permutations can also be applied to find the number of circular permutations~$Z_{\cperms}(p)$ with a given run structure~$p$.
Indeed, any circular permutation in $\cperms_{n-1}$ can be extended to a permutation in $\cperms_{n}$ by inserting $n$ at any position after the first (\eg, \cperm{14532} can be extended to \cperm{164532}, \cperm{146532}, \cperm{145632}, \cperm{145362} or \cperm{145326}).
As in the case of atomic permutations, this extension either increases the length of a run or splits a run into three runs.
Namely, we can increase the length of one run by inserting $n$ at the end or the penultimate position of an increasing run or we can split a run of length $i+j \geq 2$ into three runs of lengths $i, j$ and $1$ by inserting $n$ after $i$ elements of an increasing run or after $i+1$ elements of a decreasing run.

We introduce polynomials $C_n$ representing the run structures of all elements of $\cperms_n$, by analogy with the polynomials $A_n$ in the previous section:
\begin{equation}\label{eq:C_rel_Z}
  C_n = \sum_{p \vdash n} Z_\cperms(p) \prod_{i=1}^n x_i^{p(i)}
\end{equation}
and we say that the polynomials~$C_n$ enumerate the run structure of the circular permutations.
In the last paragraph we saw that we can use the differential operator $\mathcal{D}$ introduced in~\eqref{eq:diff} to find a recurrence relation similar to~\eqref{eq:A_recurrence}.
Namely,
\begin{subequations}\label{eq:C_recurrence}
  \begin{align}
    C_2 & \defn x_1^2, \\
    C_n & \defn \mathcal{D} C_{n-1}, \qquad (n \geq 3)
  \end{align}
\end{subequations}
giving in particular
\begin{align*}
  C_3 & = 2 x_2 x_1 \\
  C_4 & = 2 x_2^2 + 2 x_3 x_1 + 2 x_1^4 \\
  C_5 & = 2 x_4 x_1 + 6 x_3 x_2 + 16 x_1^3 x_2 \\
  C_6 & = 2 x_5 x_1 + 8 x_4 x_2 + 6 x_3^2 + 62 x_1^2 x_2^2 + 26 x_1^3 x_3 + 16 x_1^6
\end{align*}
from which we can read off that there are $2$ permutations in $\cperms_5$ corresponding to $5 = 4 + 1$, $6$ corresponding to the partition $5 = 3 + 2$ and $16$ corresponding to $5 = 2 + 1 + 1 + 1$.
As a check, we note that $6 + 16 + 2 = 24$, which is the total number of elements of $\cperms_5$; similarly, the coefficients in the expression for $C_6$ sum to $120$, the cardinality of $\cperms_6$.
More on this can be found in the last paragraph of section \ref{sec:counting_valleys}.

In summary, we have a result analogous to proposition~\ref{prop:Z_A}:
\begin{proposition}\label{prop:Z_C}
  The number $Z_\cperms(p)$ of circular permutations of length $n$ corresponding to a given run structure $p$ is determined by the polynomial $C_n$ via~\eqref{eq:C_rel_Z}.
  The polynomials~$C_n$ satisfy the recurrence relation~\eqref{eq:C_recurrence}.
\end{proposition}
\noindent
Note that circular permutations, exactly opposite to atomic permutations, always contain an even number of runs and thus $Z_\cperms(p)$ is zero for odd partitions $p$.

The enumeration of circular and atomic permutations is closely related.
In fact, introducing a generating function $\mathcal{C}$ as the formal series
\begin{equation*}
  \mathcal{C}(\lambda)
  \defn \sum_{n=0}^\infty C_{n+2} \frac{\lambda^n}{n!}
  = \sum_{n=0}^\infty \mathcal{D}^n C_2 \frac{\lambda^n}{n!}
  = \exp(\lambda \mathcal{D}) C_2,
\end{equation*}
one can show the following:
\begin{proposition}\label{prop:square}
  The formal power series $\mathcal{C}$ is the square of a formal series $\mathcal{A}$; namely,
  \begin{equation}\label{eq:square}
    \mathcal{C}(\lambda) = \mathcal{A}(\lambda)^2 = \big( \exp(\lambda \mathcal{D}) A_1 \big)^2,
  \end{equation}
  where $A_1 \defn x_1$.
\end{proposition}
\begin{proof}
  This may be seen in various ways, but the most convenient is to study the first-order partial differential equation (in infinitely many variables)
  \begin{equation}\label{eq:pde}
    \frac{\partial\mathcal{C}}{\partial \lambda} - \mathcal{D}\mathcal{C} = 0,
    \quad
    \mathcal{C}(0) = C_2
  \end{equation}
  satisfied by $\mathcal{C}$.

  We can now apply the method of characteristics to this problem.
  Since it has no inhomogeneous part, the p.d.e.~\eqref{eq:pde} asserts that $\mathcal{C}$ is constant along its characteristics.
  So, given $\lambda$ and $x_1, x_2, \dotsc$, let $\chi_1(\mu), \chi_2(\mu), \dotsc$ be solutions to the characteristic equations with $\chi_r(\lambda) = x_r$, \ie, $\chi_1(\mu), \chi_2(\mu), \dotsc$ are the characteristic curves which emanate from the point $(\lambda, x_1, x_2, \ldots)$.
  Then,
  \begin{equation*}
    \mathcal{C}(\lambda)|_{x_\bullet} = \mathcal{C}(0)|_{\chi_\bullet(0)} = C_2 \big( \chi_1(0) \big) = \chi_1(0)^2.
  \end{equation*}
  Applying the same reasoning again to $\mathcal{A}$, which obeys the same p.d.e.\ as $\mathcal{C}$ but with initial condition $\mathcal{A}(0) = A_1$,
  \begin{equation*}
    \mathcal{A}(\lambda)|_{x_\bullet} = \mathcal{A}(0)|_{\chi_\bullet(0)} = A_1 \big( \chi_1(0) \big) = \chi_1(0).
  \end{equation*}
  Therefore, proposition~\ref{prop:square} follows by patching these two equations together.
\end{proof}
\noindent
As a consequence also the polynomials $A_n$ and $C_n$ are related via
\begin{equation}\label{eq:relation}
  C_n = \sum_{m=1}^{n-1} \binom{n-2}{m-1}\, A_m A_{n-m}.
\end{equation}
It then follows that the second degree part of $C_n$ is given by
\begin{equation*}
  C^{(2)}_n = \sum_{m=1}^{n-1} \binom{n-2}{m-1}\, x_m x_{n-m}
\end{equation*}
and, applying \eqref{eq:third_degree}, that the fourth degree part can be written as
\begin{equation*}
  C_n^{(4)} = \sum_{\substack{i, j, k, l \,\geq\, 1 \\ i + j + k + l = m}} \sum_{q=1}^{k} 2\, \frac{n-l-q-1}{n-l-q-j}\, \binom{n-2}{n-l-1} \binom{n - l - q - 2}{i - 1, j - 1, k - q}\, x_i x_j x_k x_l.
\end{equation*}

Similar to the atomic permutations, we find that the alternating circular permutations satisfy (\cf \cite[\S41]{andre:1895})
\begin{equation*}
  Z_{\cperms}\bigg(\sum_{i=1}^{2n} 1\bigg) = T_n = 1, 2, 16, 272, 7936, 353792, \dotsc \quad \text{($n \geq 1$, OEIS series \oeis{A000182})}
\end{equation*}
and also $Z_{\cperms}\big(2 + \sum_{i=1}^{2n-3} 1\big) = T_n$, where $T_n$ are the \emph{tangent numbers}, the coefficients of the Maclaurin series of $\tan x = T_1 x_1 + T_2 x_3/3! + T_3 x_5/5! + \dotsb$.
Furthermore, from \eqref{eq:relation} we find the relation
\begin{equation*}
  T_{n+1} = \sum_{m=0}^n \binom{2n}{2m}\, S_m S_{n-m},
\end{equation*}
which can be traced back to $\tan' x = \sec^2 x$.

To conclude this section, we note that the argument of proposition~\ref{prop:square} proves rather more: namely, that $\exp(\lambda \mathcal{D})$ defines a ring homomorphism from $\CC[x_1, x_2, \dotsc]$ to the ring of formal power series $\CC[[x_1, x_2, \dotsc]]$. This observation can be used to accelerate computations:
for example, the fact that $A_3 = x_3 +x_1^3$ implies that
\begin{equation*}
\mathcal{A}''(\lambda) = \mathcal{A}(\lambda)^3 + \exp(\lambda\mathcal{D})x_3,
\end{equation*}
which reduces computation of $A_{n+3}=\mathcal{D}^{n+2}x_1$ to the computation of $\mathcal{D}^n x_3$. Once $\mathcal{A}$ is obtained,
we may of course determine $\mathcal{C}$ by squaring.


\section{Linear permutations}
\label{sec:linear}

In the last section we studied the run structures of circular permutations $\cperms_n$ and discovered that their run structures can be enumerated by the polynomials $A_n$.
One might ask, what the underlying reason for this is.
Circular permutations of $[n]$ have the same run structure as the linear permutations of the multiset $\{1,1,2,\dotsc,n\}$ which begin and end with $1$.
These permutations can then be split into two atomic permutations at the occurrence of their maximal element.
For example, the circular permutation \cperm{14532} can be split into the two atomic permutations \perm{145} of $\{1,4,5\}$ and \perm{5321} of $\{1,2,3,5\}$.
This also gives us the basis of a combinatorial argument for the fact that $\mathcal{C} = \mathcal{A}^2$.
Similarly it is in principle possible to encode the run structures of any subset of permutations using the polynomials $A_n$.
The goal of this section is to show how this may be accomplished for $\perms_S$ for any $S \subset \NN$.

As in sections~\ref{sec:atomic} and~\ref{sec:circular}, we want to find polynomials
\begin{equation*}
  L_n = \sum_{p \vdash n} Z_\perms(p) \prod_{i=1}^n x_i^{p(i)}
\end{equation*}
that enumerate the run structure of the permutations~$\perms_{n+1}$.
This may be achieved in a two step procedure.
Since every permutation has a unique decomposition into inextendible atomic permutations, we can enumerate the set of permutations according to this decomposition.
The enumeration of permutations by their run structure follows because the enumeration of atomic permutations has already been achieved in section~\ref{sec:atomic}.

The key to our procedure is to understand the factorisation of the run structure into those of atomic permutations. Considering $\sigma \in \perms_n$ as a word, we can write it as the concatenation $\sigma = \alpha \cdot \pi \cdot \omega$, where $\pi$ is the principal atom of~$\sigma$ (see beginning of section~\ref{sec:atomic}) and $\alpha, \omega$ are (possibly empty) subwords of~$\sigma$.
Since the decomposition of~$\sigma$ into its atoms also decomposes its run structure, the complete runs of~$\sigma$ are determined by the runs of
$\alpha \cdot 1$, $\pi$ and $n \cdot \omega$ if $\pi$ is rising, or of
$\alpha \cdot n$, $\pi$ and $1 \cdot \omega$ if $\pi$ is falling.

Let $S_\omega$ be the set of letters in $\omega$ and define $\rho: S_\omega \to S_\omega$ to be the involution mapping the $i$'th smallest element of $S_\omega$ to the $i$'th largest, for all $1 \leq i \leq |S_\omega|$.
Then the run structure of $n \cdot \omega$ is identical to that of $1 \cdot \rho(\omega)$, where $\rho(\omega)$ is obtained by applying $\rho$ letterwise to $\omega$.
Furthermore, in the case $\pi = 1 \,\dotsm\, n$, the combined run structures of $\alpha \cdot 1$ and $n \cdot \omega$ are precisely the run structure of $\alpha \cdot 1 \cdot \rho(\omega)$, while, if $\pi = n \,\dotsm\, 1$, the combined run structures of $\alpha \cdot n$ and $1 \cdot \omega$ precisely form the run structure of $\alpha\cdot n \cdot \rho(\omega)$.
We refer to $\alpha\cdot 1\cdot\rho(\omega)$ or $\alpha\cdot n\cdot\rho(\omega)$ as the \emph{residual permutation}.

Summarising, the run structure of~$\sigma$ may be partitioned into that of~$\pi$ and either $\alpha \cdot 1 \cdot \rho(\omega)$ or $\alpha \cdot n \cdot \rho(\omega)$; accordingly, the monomial for~$\sigma$ factorises into that for the principal atom~$\pi$ and that for the residual permutation.
Therefore, the polynomial enumerating linear permutations by run structure can be given in terms of the those enumerating atomic permutations of the same or shorter length and of linear permutations of strictly shorter length.

This argument can be used to give a recursion relation for~$L_n$, which enumerates permutations of $[n+1]$ by their run structure.
Taking into account that the principal atom consists of $m+1$ letters, where $1 \leq m\leq n$, of which $m-1$ may be chosen freely from the set $[2 \lddots n]$, and that it might be rising or falling, and that the residual permutation may be any linear permutation on a set of cardinality $n-m+1$, we obtain the recursion relation
\begin{equation*}
  L_n = 2 \sum_{m=1}^{n} \binom{n-1}{m-1} A_m L_{n-m},
  \qquad
  L_0=1.
\end{equation*}
Passing to the generating function,
\begin{equation*}
 \mathcal{L}(\lambda) \defn \sum_{n=0}^\infty L_n \frac{\lambda^n}{n!},
\end{equation*}
we may deduce that
\begin{equation}\label{eq:linear_de}
  \frac{\partial\mathcal{L}}{\partial\lambda} = 2\mathcal{A}(\lambda)\mathcal{L}(\lambda).
\end{equation}
Our main result in this section is:
\begin{proposition}\label{prop:linear}
  The run structure of all permutations in $\perms_{n+1}$ is enumerated by
  \begin{equation}\label{eq:linear}
    L_n = \sum_{p \vdash n} \frac{2^{\abs{p}}}{\ord p}\, \binom{n}{p} \prod_{i=1}^{\abs{p}} A_{p_i},
    \quad
    L_0 = 1,
  \end{equation}
  where the sum is over all partitions $p = p_1 + p_2 + \dotsb$ of~$n$, $\abs{p}$ is the number of parts of partition, $\ord p$ is the symmetry order of the parts of~$p$ (\eg, for $p = 1+1+2+3+3$ we have $\ord p = 2! 2!$) and $\binom{n}{p}$ is the multinomial with respect to the parts of~$p$.
  The generating function for the~$L_n$ is
  \begin{equation}\label{eq:linear_gen}
    \mathcal{L}(\lambda) \defn \sum_{n=0}^\infty L_n \frac{\lambda^n}{n!}
    = \exp \left( 2 \int_0^\lambda \mathcal{A}(\mu)\, \dif\mu\right).
  \end{equation}
\end{proposition}
\begin{proof}
  Equation~\eqref{eq:linear_gen} follows immediately from \eqref{eq:linear_de}, as $\mathcal{L}(0)=1$, whereupon Faà di Bruno's formula~\cite[Eq.~(1.4.13)]{olver:2010} yields~\eqref{eq:linear}.
\end{proof}

To conclude this section, we remark that the first few $L_n$ are given by
\begin{align*}
  L_1 & = 2 A_1 \\
  L_2 & = 4 A_1^2 + 2 A_2 \\
  L_3 & = 8 A_1^3 + 12 A_1 A_2 + 2 A_3 \\
  L_4 & = 16 A_1^4 + 48 A_1^2 A_2 + 12 A_2^2 + 16 A_1 A_3 + 2 A_4 \\
  L_5 & = 32 A_1^5 + 160 A_1^3 A_2 + 120 A_1 A_2^2 + 80 A_1^2 A_3 + 40 A_2 A_3 + 20 A_1 A_4 + 2 A_5\\
  L_6 &= 64 A_1^6 + 480 A_1^4 A_2 + 320 A_1^3 A_3 + 720 A_1^2 A_2^2 + 120 A_1^2 A_4 + 480 A_1 A_2 A_3 + 120 A_2^3 \\&\qquad + 24 A_1 A_5 +
60 A_2 A_4 + 40 A_3^2 + 2A_6.
\end{align*}
Expanding the $A_k$ and writing the $L_n$ instead in terms of $x_i$, we obtain from these
\begin{align*}
  L_1 & = 2 x_1 \\
  L_2 & = 4 x_1^2 + 2 x_2 \\
  L_3 & = 10 x_1^3 + 12 x_1 x_2 + 2 x_3 \\
  L_4 & = 32 x_1^4 + 58 x_1^2 x_2 + 12 x_2^2 + 16 x_1 x_3 + 2 x_4 \\
  L_5 & = 122 x_1^5 + 300 x_1^3 x_2 + 142 x_1 x_2^2 + 94 x_1^2 x_3 + 40 x_2 x_3 + 20 x_1 x_4 + 2 x_5\\
  L_6 & = 544 x_1^6 + 1682 x_1^4 x_2 + 568x_1^3 x_3 + 1284 x_1^2 x_2^2 + 138 x_1^2 x_4 + 556 x_1 x_2 x_3 + 142 x_2^3 \\&\qquad + 24 x_1 x_5 + 60 x_2 x_4 + 40 x_3^2 + 2x_6,
\end{align*}
which show no obvious structure, thereby making proposition \ref{prop:linear} that much more remarkable.


\section{Counting valleys}
\label{sec:counting_valleys}

Instead of enumerating permutations by their run structure, we can count the number of valleys of a given (circular) permutation.
Taken together, the terms $C_n$ involving a product of $2k$ of the $x_i$ relate precisely to the circular permutations $\cperms_n$ with $k$ valleys.
Since any circular permutation in $\cperms_n$ can be understood as a permutation of $[3 \lddots n+1]$ with a prepended $1$ and an appended $2$ (\cf beginning of section~\ref{sec:linear}), $C_n$ may also be used to enumerate the valleys of ordinary permutations of $[n-1]$.
Namely, terms of $C_{n+1}$ with a product of $2(k+1)$ variables $x_i$ relate to the permutations of $\perms_n$ with $k$ valleys (\ie, terms of $L_{n+1}$ which are a product of $2k$ of the $x_i$).

Let $V(n, k)$ count the number of permutations of $n$ elements with $k$ valleys.
Then we see that the generating function for $V(n, k)$ for each fixed $n \geq 1$ is
\begin{equation*}
  K_n(\kappa) \defn \sum_{k=1}^n \kappa^k V(n, k)
  = \frac{1}{\kappa} C_{n+1}(\sqrt{\kappa}, \dotsc, \sqrt{\kappa})
\end{equation*}
and we define $K_0(\kappa) \defn 1$.
The first few $K_n$ are
{\savebox\strutbox{$\vphantom{\kappa^2}$}
\begin{align*}
  K_1(\kappa) & = 1 \\
  K_2(\kappa) & = 2 \\
  K_3(\kappa) & = 4 + 2 \kappa \\
  K_4(\kappa) & = 8 + 16 \kappa \\
  K_5(\kappa) & = 16 + 88 \kappa + 16 \kappa^2 \\
  K_6(\kappa) & = 32 + 416 \kappa + 272 \kappa^2,
\end{align*}}%
which coincide with the results in~\cite{rieper:2000}.
In particular, the constants are clearly the powers of $2$, the coefficients of $\kappa$ give the sequence \oeis{A000431} of the OEIS~\cite{oeis} and the coefficients of $\kappa^2$ are given by the sequence \oeis{A000487}.
Likewise, the coefficients of $\kappa^3$ may be checked against the sequence \oeis{A000517}.
In fact, the same polynomials appear in André's work, in which he obtained a generating function
closely related to \eqref{eq:kitaev} below; see \cite[\S158]{andre:1895} (his final formula contains a number of sign errors, and is given in a form in which all quantities are real for $\kappa$ near $0$; there is also an offset, because his polynomial $A_n(\kappa)$ is our $K_{n-1}(\kappa)$).

\begin{proposition}
  The bivariate generating function, \ie, the generating function for arbitrary $n$, is
  \begin{equation*}
    \mathcal{K}(\nu, \kappa) = \sum_{n=0}^\infty K_n(\kappa) \frac{\nu^n}{n!} = 1 + \frac{1}{\kappa} \int_0^\nu \mathcal{C}(\mu)|_{x_1 = x_2 = \dotsb = \sqrt{\kappa}}\; \dif\mu
  \end{equation*}
  and is given in closed form by
  \begin{equation}\label{eq:kitaev}
    \mathcal{K}(\nu, \kappa) = 1 - \frac{1}{\kappa} + \frac{\sqrt{\kappa - 1}}{\kappa} \tan\big( \nu \sqrt{\kappa - 1} + \arctan(1/\sqrt{\kappa - 1}) \big).
  \end{equation}
\end{proposition}
\noindent
This result was found by Kitaev~\cite{kitaev:2007} and in the remainder of this section we will show how it may be derived from the recurrence relation~\eqref{eq:C_recurrence} of $C_n$.

To this end, we first note that $C_{n+1}$ satisfies the useful scaling relation
\begin{equation*}
  \lambda^{n+1} C_{n+1}(x_1, x_2, \dotsc, x_n) = C_{n+1}(\lambda x_1, \lambda^2 x_2, \dotsc, \lambda^n x_n).
\end{equation*}
Setting $x_i = x / \lambda = \sqrt{\kappa}$ for all $i$, this implies
\begin{equation*}
  \lambda^{n+1} C_{n+1}(\sqrt{\kappa}, \dotsc, \sqrt{\kappa}) = C_{n+1}(x, \lambda x, \dotsc, \lambda^{n-1} x)
\end{equation*}
and we find, by inserting the recurrence relations~\eqref{eq:C_recurrence} and applying the chain rule, that with this choice of variables
\begin{equation*}
  \frac{1}{x^2} C_{n+1}(x, \lambda x, \dotsc, \lambda^{n-1} x) = \lambda x \pd{C_n}{x} + x^2 \pd{C_n}{\lambda} + 2 \lambda C_n.
\end{equation*}
Hence, in turn, $K_n(\kappa) = \kappa^{-1} C_{n+1}(\sqrt{\kappa}, \dotsc, \sqrt{\kappa})$ satisfies the recurrence relation
\begin{equation}\label{eq:Kn}
  K_n(\kappa) = 2 \kappa (1 - \kappa) K_{n-1}'(\kappa) + \big( 2 + (n - 2) \kappa \big) K_{n-1}(\kappa)
\end{equation}
for $n \geq 2$.
For the bivariate generating function $\mathcal{K}$ this, together with $K_0 = K_1 = 1$, implies the p.d.e.
\begin{equation*}
  (1 - \nu \kappa) \pd{\mathcal{K}}{\nu} + 2 \kappa (\kappa - 1) \pd{\mathcal{K}}{\kappa} + (\kappa - 2) \mathcal{K} = \kappa - 1,
\end{equation*}
which is to be solved subject to the initial condition $\mathcal{K}(0, \kappa) = 1$.

The above equation may be solved as follows: first, we note that there is a particular integral $1 - 1/\kappa$, so it remains to solve the homogeneous equation.
In turn, using an integrating factor, the latter may be rewritten as
\begin{equation}\label{eq:homog}
  (1 - \nu \kappa) \pd{}{\nu} \frac{\kappa \mathcal{K}}{\sqrt{\kappa - 1}} + 2 \kappa (\kappa - 1) \pd{}{\kappa} \frac{\kappa \mathcal{K}}{\sqrt{\kappa - 1}} = 0,
\end{equation}
for which the characteristics obey
\begin{equation*}
  \od{\nu}{\kappa} = \frac{1 - \nu \kappa}{2 \kappa (\kappa - 1)}.
\end{equation*}
Solving this equation, we find that
\begin{equation*}
  \nu \sqrt{\kappa - 1} + \arctan \frac{1}{\sqrt{\kappa - 1}} = \text{const}
\end{equation*}
along characteristics; as~\eqref{eq:homog} asserts that $\kappa \mathcal{K}/\sqrt{\kappa - 1}$ is constant on characteristics, this gives
\begin{equation*}
  \mathcal{K}(\nu, \kappa) = 1 - \frac{1}{\kappa} + \frac{\sqrt{\kappa - 1}}{\kappa} f\big( \nu \sqrt{\kappa - 1} + \arctan(1/\sqrt{\kappa - 1}) \big)
\end{equation*}
for some function $f$.
Imposing the condition $\mathcal{K}(0, \kappa) = 1$, it is plain that $f = \tan$, and we recover Kitaev's generating function~\eqref{eq:kitaev}.

To close this section, we note that \eqref{eq:Kn} has the consequence that $K_n(1) = n K_{n-1}(1)$ for all $n \geq 2$ and hence that $C_{n+1}(1,\dotsc,1) = K_n(1) = n!$ for such $n$, and indeed all $n\geq 1$, because $C_2(1, 1) = K_1(1) = 1$. The generating function obeys
\begin{equation*}
  \mathcal{C}(\lambda)|_{x_\bullet=1}
  = \sum_{n=0}^\infty (n+1)! \frac{\lambda^n}{n!}
  = (1-\lambda)^{-2}
\end{equation*}
for all non-negative $\lambda < 1$ from which it also follows that
\begin{equation}\label{eq:geom_A}
  \mathcal{A}(\lambda)|_{x_\bullet=1} = (1-\lambda)^{-1}
\end{equation}
(as $A_1(1) = 1$, we must take the \emph{positive} square root) and hence $A_{n}|_{x_\bullet=1} = (n-1)!$ for all $n \geq 1$.
This gives a consistency check on our results: the coefficients in the expression for $A_n$ sum to $(n-1)!$, the cardinality of $\aperms^\pm_{n+1}$, while those in $C_n$ sum to the cardinality of $\cperms_n$.
Furthermore, inserting \eqref{eq:geom_A} into the generating function $\mathcal{L}(\lambda)$ in \eqref{eq:linear_gen}, we find
\begin{equation*}
  \mathcal{L}(\lambda)|_{x_\bullet=1}
  = \sum_{n=0}^\infty L_n(1,\dotsc,1) \frac{\lambda^n}{n!}
  = (1-\lambda)^{-2},
\end{equation*}
and thus $L_{n+1}(1,\dotsc,1) = n!$, which is the cardinality of $\perms_n$.


\section{Other applications}
\label{sec:app}

The original motivation for this work arose in quantum field theory, in computations related to the probability distribution of measurement outcomes for quantities such as averaged energy densities~\cite{fewster:2012b}. One actually
computes the cumulants $\kappa_n$ ($n \in \NN$) of the distribution: $\kappa_1 = 0$, while for each $n \geq 2$, $\kappa_n$ is given as a sum indexed by circular permutations $\sigma$ of $[n]$ such that $\sigma(1) = 1$ and $\sigma(2) < \sigma(n)$, in which each permutation contributes a term that is a multiplicative function of its run structure:
\begin{equation*}
  \kappa_n = \sum_{\sigma} \Phi(\sigma)
\end{equation*}
where $\Phi(\sigma)$ is a product over the runs of $\sigma$, with each run of length $r$ contributing a factor $y_r$. Owing to the restriction $\sigma(2) < \sigma(n)$, precisely half of the circular permutations are admitted, and so $\kappa_n=\frac{1}{2}C_n(y_1,y_2,\ldots,y_n)$.
Thus the cumulant generating function is
\begin{align*}
  W(\lambda)
  &\defn \sum_{n=2}^\infty \kappa_n \frac{\lambda^n}{n!} = \frac{1}{2} \int_0^\lambda \dif\mu\, (\lambda - \mu) \mathcal{C}(\mu)|_{x_\bullet=y_\bullet} \\
  &= \frac{1}{2} \int_0^\lambda \dif\mu\, (\lambda - \mu) \mathcal{A}(\mu)|_{x_\bullet=y_\bullet}^2
\end{align*}
and the moment generating function is $\exp W(\lambda)$ in the usual way. This expression makes sense a formal power series, but also as
a convergent series within an appropriate radius of convergence.
The values of $y_n$ depend on the physical quantity involved and the way it is averaged.
In one case of interest
\begin{align*}
  y_n
  &= 8^n \int_{(\RR^+)^{\times n}} \dif k_1\, \dif k_2\, \cdots\, \dif k_n\, k_1 k_2 \,\cdots\, k_n  \exp \left[ -k_1-\left(\sum_{i=1}^{n-1} |k_{i+1}-k_i|\right)-k_{n} \right] \\
  &= 2^n \sum_{r_{n-1}=0}^{2} \sum_{r_{n-2}=0}^{2+r_{n-1}} \cdots \sum_{r_1=0}^{2+r_2} \prod_{k=1}^{n-1} (1+r_k) \\
  &=2, 24, 568, 20256, 966592, \dotsc \qquad (n \geq 1)
\end{align*}
(the sums of products must be interpreted as an overall factor of unity in the case $n=1$).
Numerical investigation leads to a remarkable identity
\begin{equation*}
  \mathcal{A}(\lambda)|_{x_\bullet=y_\bullet} = \frac{2}{1-12\lambda} \qquad \textrm{(conjectured)}
\end{equation*}
with exact agreement for all terms so far computed (checked up to $n=65$).
We do not have a proof for this statement, but the conjecture seems fairly secure. For example, we have shown above that $A_5= x_5 + 7 x_3 x_1^2 + 11 x_2^2 x_1 + 5 x_1^5$; substituting for $x_n$ the values of $y_n$ obtained above, we find $A_5= 995328$ which coincides with the fourth order coefficient in the expansion
\begin{equation*}
\frac{2}{1-12\lambda}=  2+ 24\lambda + 576\frac{\lambda^2}{2!}+20736\frac{\lambda^3}{3!} +995328\frac{\lambda^4}{4!}+O(\lambda^5).
\end{equation*}
In~\cite{fewster:2012b}, this conjecture was used to deduce
\begin{equation*}
  \exp\big(W(\lambda)\big) = \e^{-\lambda/6}(1-12\lambda)^{-1/72} \qquad \textrm{(conjectured)},
\end{equation*}
which is the moment generating function of a shifted Gamma distribution.
The other generating functions of interest, with these values for the $x_k$ are
\begin{equation*}
  \mathcal{C}(\lambda)|_{x_\bullet=y_\bullet} = \frac{4}{(1-12\lambda)^2},
  \qquad
  \mathcal{L}(\lambda)|_{x_\bullet=y_\bullet} = (1-12\lambda)^{-1/3}
  \qquad \textrm{(conjectured)}.
\end{equation*}
For example, we have $C_5 =  2 x_4 x_1 + 6 x_3 x_2 + 16 x_1^3 x_2  =
165888$ and $L_5 = 122 x_1^5 + 300 x_1^3 x_2 + 142 x_1 x_2^2 + 94 x_1^2 x_3 + 40 x_2 x_3 + 20 x_1 x_4 + 2 x_5 =3727360$, to be compared with the terms of order $\lambda^3$ and $\lambda^5$, respectively, in
the expansions
\begin{align*}
\frac{4}{(1-12\lambda)^2} &= 4+ 96\lambda+ 3456\frac{\lambda^2}{2!}+ 165888\frac{\lambda^3}{3!}+ 995328\frac{\lambda^4}{4!}+ O(\lambda^5), \\
 (1-12\lambda)^{-1/3} &=
1+ 4\lambda+ 64\frac{\lambda^2}{2!}+ 1792\frac{\lambda^3}{3!}+ 71680\frac{\lambda^4}{4!}+ 3727360\frac{\lambda^5}{5!}+O(\lambda^6).
\end{align*}

A natural question is whether there are other sequences that can be substituted for the $x_k$ to produce generating functions with simple closed forms.
To close, we give three further examples, with the corresponding generating functions computed.
The first has already been encountered in section \ref{sec:counting_valleys} and corresponds to the case $x_k = 1$ for all $k \in \NN$.

The second utilizes the alternating Catalan numbers: setting
\begin{equation*}
  x_{2k+1}= \frac{(-1)^k}{k+1} \binom{2k}{k}, \quad(k \geq 0),
  \qquad
  x_{2k}=0, \quad (k \geq 1)
\end{equation*}
and thus $A_{2k} = 0$, we obtain, again experimentally,
\begin{equation*}
  \mathcal{C}(\lambda) = \mathcal{A}(\lambda) = 1,
  \quad
  \mathcal{L}(\lambda) = \e^{2\lambda}
  \qquad \textrm{(conjectured)}
\end{equation*}
with exact agreement checked up to permutations of length $n = 65$.
For example, one sees easily that with $x_1=1$, $x_3=-1$, $x_5=2$ and $x_2=x_4=0$, the expressions $A_k$ and $C_k$ given in sections~\ref{sec:atomic} and~\ref{sec:circular} vanish for $2\le k\le 6$,
and have $A_1=C_1=1$, likewise, $L_k=2^k$ for $1\le k\le 6$.

Third, André's classical result on alternating permutations (\cf last and penultimate paragraph of section \ref{sec:atomic} and \ref{sec:circular} respectively) gives the following: setting
\begin{equation*}
  x_1=1, \qquad
  x_{k}=0, \quad (k \geq 2)
\end{equation*}
we have, using \eqref{eq:square} and \eqref{eq:linear_gen},
\begin{equation*}
  \mathcal{A}(\lambda) = \sec\lambda,
  \quad
  \mathcal{C}(\lambda)=\sec^2\lambda,
  \quad
  \mathcal{L}(\lambda) = (\sec\lambda + \tan\lambda)^2.
\end{equation*}
It seems highly likely to us that many other examples can be extracted from the structures we have described.

Moreover, we remark that it is possible to implement a merge-type sorting algorithm, called \emph{natural merge sort} \cite[Chap. 5.2.4]{knuth:1998}, based upon splitting permutations of an ordered set $S$ into its runs, which are ordered (alternatingly in ascending and descending order) sequences $S_i \subset S$.
Repeatedly merging these subsequences, one ultimately obtains an ordered sequence.
For example, first, we split the permutation \perm{542368719} into \perm{542}, \perm{368}, \perm{71} and \perm{9}.
Then, we reverse every second sequence (depending on whether the first or the second sequence is in ascending order): \perm{542} $\mapsto$ \perm{245} and \perm{71} $\mapsto$ \perm{17}.
Depending on the implementation of the merging in the following step, this `reversal' step can be avoided.
Last, we merge similarly to the standard merge sort: \perm{245} $\vee$ \perm{368} $\mapsto$ \perm{234568}, \perm{17} $\vee$ \perm{9} $\mapsto$ \perm{179} and finally \perm{234568} $\vee$ \perm{179} $\mapsto$ \perm{123456789}.
Natural merge sort is a fast sorting algorithm for data with preexisting order.
Using the methods developed above to enumerate permutations by their run structure, it is in principle possible to give average (instead of best- and worst-case) complexity estimates for such an algorithm.


\begin{acknowledgments}
  D.~S. is very grateful for the kind hospitality extended by the Department of Mathematics at the University of York where this work was carried out.
  C.~J.~F. thanks Nik Ruškuc and Vincent Vatter for useful comments on an early version of the text.
\end{acknowledgments}


\end{document}